                \documentclass[11pt,letterpaper]{amsart}
                \usepackage{amsmath}
\usepackage{amsfonts}
\usepackage{amssymb}
\usepackage{amsthm}
\usepackage{url}
\usepackage[all]{xy}
\usepackage{dsfont}
\usepackage{graphicx}
\usepackage{caption}
\usepackage{subcaption}
\usepackage{comment} 
\usepackage{stmaryrd}
\usepackage{hyperref}
\usepackage{todonotes}
\usepackage{color}
\usepackage{enumerate,tikz-cd}
\usepackage{fixltx2e}
\usepackage{MnSymbol}
\usepackage{comment}
\allowdisplaybreaks

\numberwithin{equation}{section}

\newtheorem{proposition}{Proposition}[section]

\newtheorem{lemma}[proposition]{Lemma}
\newtheorem{theorem}[proposition]{Theorem}
\newtheorem{corollary}[proposition]{Corollary}

\theoremstyle{definition}
\newtheorem{remark}[proposition]{Remark}
\newtheorem{definition}[proposition]{Definition}
\newtheorem{example}[proposition]{Example}

\DeclareMathOperator{\GL}{GL}

\DeclareMathOperator{\Aut}{Aut}

\DeclareMathOperator{\SL}{SL}

\DeclareMathOperator{\Spec}{Spec}

\DeclareMathOperator{\DF}{DF}

\DeclareMathOperator{\ord}{ord}

\DeclareMathOperator{\Hilb}{Hilb}

\DeclareMathOperator{\CM}{CM}
\DeclareMathOperator{\ses}{ss}
\DeclareMathOperator{\cscK}{cscK}

\newcommand{\N}{\mathbb{N}}

\newcommand{\C}{\mathbb{C}}

\newcommand{\Z}{\mathbb{Z}}

\newcommand{\Q}{\mathbb{Q}}

\newcommand{\pr}{\mathbb{P}}
\renewcommand{\epsilon}{\varepsilon}

\newcommand{\scO}{\mathcal{O}}

\renewcommand{\L}{\mathcal{L}}
\newcommand{\X}{\mathcal{X}}

\renewcommand{\H}{\mathcal{H}}
\newcommand{\U}{\mathcal{U}}

\renewcommand{\phi}{\varphi}

\title[The cscK condition is very general]{The constant scalar curvature K\"ahler condition is very general}

\author[Ruadha\'i Dervan]{Ruadha\'i Dervan}

\address{Ruadha\'i Dervan, Warwick Mathematics Institute, Zeeman Building
University of Warwick, Coventry CV4 7AL United Kingdom}\email{ruadhai.dervan@warwick.ac.uk}

\begin{document}

\begin{abstract} Recent work of Trusiani implies that the existence of a constant scalar curvature K\"ahler metric on a smooth polarised variety with discrete automorphism group is equivalent to uniform arc K-stability. We prove that uniform arc K-stability  is essentially algebraic in flat families of polarised varieties. When the polarised varieties are further smooth and have discrete automorphism group, combining these two results implies that the constant scalar curvature K\"ahler locus is very general. 

We use this result to give examples of constant scalar curvature K\"ahler metrics whose existence only follows from the recent solution of the Yau--Tian--Donaldson conjecture. Our technique is to prove a general result stating that stability of a pair in the sense of Paul is essentially an algebraic property in families, and to employ prior work with Reboulet relating uniform arc K-stability to stability of an associated pair.

\end{abstract}

\maketitle

\section{Introduction}

Consider a smooth projective variety $X$ endowed with an ample line bundle $L$, such that the automorphism group $\Aut(X,L)$ is finite. The Yau--Tian--Donaldson conjecture states that the existence of a constant scalar curvature K\"ahler (cscK) metric in $c_1(L)$ is captured by the underlying algebraic geometry of $(X,L)$  \cite{STY, GT1,SD1}. This conjecture has recently been solved through works of Boucksom--Jonsson, Darvas--Zhang, Li and Trusiani \cite{BJ,DZ,CL,AT}, which provide various algebro-geometric stability conditions which are equivalent to the existence of a cscK metric.

Another way of viewing the algebraicity of the cscK condition, as in Donaldson \cite{SD3}, is to consider the behaviour of the cscK condition in families. We use the solution of the Yau--Tian--Donaldson conjecture due to Trusiani to prove that this behaviour is essentially algebraic. Let $(\X,\L) \to B$ be a flat family of smooth polarised varieties with finite automorphism group (so $\L$ is relatively ample). Our main result is as follows.

\begin{theorem}\label{intromainthm}The cscK locus  is very general in $B$.
\end{theorem}

That is, if we set $$B_{\cscK}= \left\{b \in B: c_1(\L_b) \textrm{ admits a cscK metric}\right\},$$ where $(\X_b,\L_b)$ is the fibre over $b\in B$, then $B_{\cscK}$ is  a countable intersection of Zariski open subsets of $B$; we allow the possibility that $B_{\cscK}$ is empty, namely the possibility that no fibre admits a cscK metric. Provided the cscK locus is nonempty, it follows in particular that $B_{\cscK}$ is dense in $B$. We also note that the cscK locus is classically known to be open in the analytic topology, though this is not sufficient to prove the conjectural statement that the cscK locus is actually Zariski open (considering for example the set  $\{1/n, n\in \N\}\cup \{0\} \subset\C$ whose complement is very general and analytically open, but not Zariski open), and also note that the restriction on the automorphism groups of the fibres is necessary \cite{GT1}. We also mention recent work of Zhang--Zhou, which proves a stronger Zariski openness result for the cscK equation on projective surfaces, provided the polarisation is sufficiently close to the anticanonical class in a numerical sense \cite{ZZ}, extending ideas of Donaldson in the K\"ahler--Einstein setting \cite{SD3}.

Theorem \ref{intromainthm} gives a new tool in the construction of cscK manifolds. For example, Arezzo--Della Vedova--Shi give a recipe for producing constant scalar curvature K\"ahler metrics on finite covers \cite{ADVS} (applicable in simple cases such as branched covers of $\pr^1\times\pr^1$), and such finite covers admit a finite group of symmetry. Typically, generic deformations of such finite covers break this symmetry and so do not fall under the results of  \cite{ADVS} (see for example Cynk--van Straten  \cite{CVS} for general results on deformations of double covers). It thus follows from Theorem \ref{intromainthm} that very general deformations of covers of $\pr^1\times\pr^1$ admit cscK metrics, which was not previously known. Our technique employs crucially the recent solution of the Yau--Tian--Donaldson conjecture, and these examples are the first new examples of cscK manifolds whose existence only follows from the resolution of this conjecture. 

We  briefly recall the history of related examples, before explaining the ideas behind Theorem \ref{intromainthm}. The first examples of cscK metrics which were produced by an instance of the Yau--Tian--Donaldson conjecture are toric surfaces \cite{wang-zhou-toric}, using the solution in this case of Donaldson \cite{donaldson:toric}. Assuming discrete automorphism group, the first K-stable examples not then known to admit a cscK metric were produced by the first author, using $\alpha$-invariant techniques \cite{RD-alpha}. Subsequently, it was proven directly that these examples admit cscK metrics by Chen--Cheng \cite{chencheng:iiiexistence}, using that the Mabuchi functional is coercive in this case \cite{RD:mabuchi}. There have been a great many examples constructed employing the solution of the Yau--Tian--Donaldson conjecture for K\"ahler--Einstein metrics on Fano manifolds, beginning with work of Ilten--S\"uss \cite{IS}. We refer to \cite{calabi} for many such examples of Fano threefolds, and for further developments.

The work of Trusiani implies that the existence of a cscK metric on $(X,L)$ is equivalent to uniform K-stability of $(X,L)$, in the sense of \cite{uniform, BHJ}. By prior work of the author and Reboulet, the existence of a cscK metric implies uniform arc K-stability (a notion essentially due to Donaldson \cite{SD2}, see also Wang \cite{XW}). Arcs, or models, provide more general degenerations of $(X,L)$ than the test configurations involved in uniform K-stability. As such, uniform arc K-stability implies uniform K-stability. Since this in turn implies the existence of a cscK metric, it follows by combining these results that the existence of a cscK metric is equivalent to uniform arc K-stability. Theorem \ref{intromainthm} is therefore a consequence of the following general result around the behaviour of uniform K-stability with respect to arcs in families. Consider a flat family $(\X,\L)$ of polarised varieties, which we do not assume are smooth.

\begin{theorem}\label{intromainthm2}The uniformly arc K-stable  locus is a countable union of countable intersections of constructible sets.
\end{theorem}

Although the statement is technical, the spirit of Theorem \ref{intromainthm2} is that the uniformly arc K-stable locus has an algebraic structure in $B$, and this is sufficient for applications. We also prove that the arc K-semistable locus is a countable intersection of constructible sets. To parse the result, recall that a locally closed subset of $B$ (in the Zariski topology) is an open subset of a closed subset, while a constructible set is a finite union of locally closed subsets.  As the cscK condition is classically open in the analytic topology under a discrete automorphism group assumption \cite{FS}, Theorem \ref{intromainthm2} implies Theorem \ref{intromainthm} by applying the work of Trusiani \cite{AT}. We emphasise that the main point in the proof of Theorem \ref{intromainthm} is to use the version of the proof of the Yau--Tian--Donaldson conjecture due to Trusiani, and our contribution is to clarify the sense in which the uniform arc K-stability condition is algebraic, through Theorem \ref{intromainthm2}. 

In contrast with Theorem \ref{intromainthm}, our proof of Theorem \ref{intromainthm2} is completely algebro-geometric. To put this result into context, we note that algebraicity results are used in giving moduli spaces algebraic structures, and Theorems \ref{intromainthm} and \ref{intromainthm2} provide results in this direction. We note that a moduli space of cscK manifolds with fixed numerical invariants is known to exist as a complex space  \cite{FS, DN}, whereas its existence as an algebraic space (or moreover as a quasi-projective scheme) remains conjectural.  

In the Fano case, the construction of a moduli space of K-polystable Fano varieties is a central result of the field proven over the last decade \cite{CX3,LXZ}, and stronger results than Theorems \ref{intromainthm} and \ref{intromainthm2} play an important role in this direction. The stronger analogue of Theorem \ref{intromainthm}, namely Zariski openness of the K\"ahler--Einstein locus in families of Fano manifolds with discrete automorphism group, is proven in \cite{SD3,YO}. Purely algebro-geometric analogues of Theorem  \ref{intromainthm2} (which are, again, stronger)  in the Fano case are due to Blum--Liu \cite[Theorem 3]{BL} and Codogni--Patakfalvi \cite[Section 4.4]{CP}, who proved that K-semistability of Fano varieties is a very general property, where we note that K-semistability of a Fano variety is equivalent to K-semistability \cite[Section 5]{BX} (see also \cite[Remark 7.9]{CX2}). These  results were later improved by Blum, Liu and Xu to the stronger statement that K-semistability and uniform  K-stability are actually Zariski open properties in families of Fano varieties \cite{BL2,BLX,CX}

To prove Theorem \ref{intromainthm2}, we employ ideas of Paul \cite{SP}, alongside further developments of the author and Reboulet \cite{DR}. These works imply that uniform arc K-stability is equivalent to stability of a sequence of pairs of points, where stability of pairs has been introduced by Paul as a generalisation of geometric invariant theory. In Section \ref{non-ample}, we prove generalisations of some of these results, which apply without any restrictions on the singularities of the varieties involved. As exploited there, the perspective we take is that stability of pairs is the analogue of stability in the sense of geometric invariant theory, for group actions on projective schemes of the form $(Y,H)$, where $H$ is not assumed ample (this being the classical setting when $H$ is ample), though this is only a special case to which Paul's theory applies. When the varieties under consideration are smooth, it is most natural to associate the pair of the Chow point and the discriminant point, as in Paul \cite{SP}; to obtain results in the singular setting, we instead use the Hilbert scheme, and write the CM line bundle as a difference of ample line bundles. Having reduced to a question around stability of pairs, the key result we prove is the following.

\begin{theorem}\label{introsecondthm} The locus of stable pairs is a countable union of constructible sets. The semistable locus is constructible. \end{theorem}

Tian has conjectured that one can form moduli spaces of stable pairs \cite[Conjecture 8.1]{GT2}. Theorem \ref{introsecondthm} thus proves one part of his conjecture, in a weakened form. Through the recent progress on the Yau--Tian--Donaldson conjecture, this would allow the construction of moduli of cscK (or uniformly K-stable) smooth polarised varieties. In this direction, we also note that it was proven in \cite{SP, DR} that uniform arc K-stability implies finiteness of the automorphism group, which is another ingredient in the construction of moduli. It would be interesting to understand the extent to which one can further develop the moduli theory surrounding stability of pairs. 

We also advertise that it would be of interest to understand the precise behaviour of uniform arc K-polystability in families, by analogy with the work of Ortu in the cscK setting, which explains this behaviour in the analytic topology \cite{AO}, where the cscK condition is not an open condition (even in the analytic topology) \cite{GT1}.

Theorem \ref{introsecondthm} is a variant of the classical result that semistability in the sense of  geometric invariant theory is a Zariski open property, which is proven through the existence of invariant nonvanishing sections involved in that setting; such an approach is unavailable in our setting.

We note that uniform K-stability is motivated by the Hilbert--Mumford criterion, involving test configurations as analogues of one-parameter subgroups, and the CM line bundle. This was the perspective of Odaka \cite{YO}, in his proof that the K\"ahler--Einstein condition is Zariski open in flat families of Fano manifolds with discrete automorphism group. The approach of Odaka is to write the CM line bundle as a difference of ample line bundles, and use constructibility properties of GIT weights originating with Mumford; equivalent results were established by Paul \cite[Theorem 1.3 (4)]{SP2} (see also Li--Wang--Xu \cite[Section 9]{LWX}). The work of Odaka and Paul invokes the convex geometry of GIT stability more prominently; our approach, by contrast, is more geometric and centres the theory of stability of pairs. It seems likely that one could adapt their ideas to uniform K-stability, which would require a somewhat different approach to incorporate the minimum norm, which does not naturally arise from a line bundle on the Hilbert scheme. It would be interesting to do so, as this would prove an approach to Theorem \ref{intromainthm} which avoids the theory of stability of pairs. Our original goal, in any case, was to prove Theorem \ref{introsecondthm}, around the theory of stability of pairs, and we view Theorem \ref{intromainthm} as an application of these ideas. 

We end by noting that the ideas of Theorem \ref{introsecondthm} should be very widely applicable, beyond the cscK problem. For example, mixed Monge--Amp\`ere equations and their stability counterparts can be interpreted through stability of pairs \cite{RR}, similarly for the condition for the $k\textsuperscript{th}$-Chern--Weil representative of the tangent bundle to be harmonic \cite{QW}, and we expect many further examples.

\subsection*{Notation} We work over the complex numbers, and do not assume varieties are irreducible. A polarisation of a projective scheme is a choice of ample $\Q$-line bundle. We denote by $\C\llparenthesis t\rrparenthesis$ the field of Laurent series and $\C\llbracket t\rrbracket$ the ring of formal power series.

\subsection*{Acknowledgements} I thank  R\'emi Reboulet for many discussions on stability of pairs, and Harold Blum, Theo Papazachariou, Sean Paul and Junsheng Zhang for helpful comments. I was funded by a Royal Society University Research Fellowship (URF$\backslash$R1$\backslash$20104) for the duration of this work.

\section{Preliminaries}
\subsection{Arc K-semistability}\label{sec:arcs-prelims}
Let $(X,L)$ be polarised variety of dimension $n$. We recall the theory of uniform arc K-stability, due to Donaldson \cite{SD2} (see also \cite[Section 2]{DR} and \cite{XW}).

\begin{definition} A \emph{model} is a collection $\pi: (\X,\L) \to \Spec\C\llbracket t\rrbracket$, where
\begin{enumerate}[(i)]
\item $\pi$ is a surjective, flat morphism from the scheme $\X$ to $\Spec\C\llbracket t\rrbracket$,
\item $\L$ is a relatively ample line bundle,
\end{enumerate}
 along with an identification $(\X_{\C\llparenthesis t\rrparenthesis},\L_{\C\llparenthesis t\rrparenthesis})\cong (X_{\C\llparenthesis t\rrparenthesis}, rL_{\C\llparenthesis t\rrparenthesis})$ (with these denoting the restriction and product respectively). We call $r$ the \emph{exponent} of the model.
\end{definition}

Models generalise test configurations, which are models induced by $\C^*$-degenerations of $(X,L)$ \cite{SD1,GT1}. We associate numerical invariants to models, in the following manner, much as in \cite{SD2}. By flatness and relative ampleness, the pushforward $\pi_*(k\L)$ is a vector bundle for $k\gg 0$, and we obtain a sequence of line bundles over $\Spec\C\llbracket t\rrbracket$ through the (polynomial) Knudsen--Mumford expansion \cite[Theorem 4]{KM} (see also \cite[Appendix 5.D]{GIT}) $$\det \pi_*(k\L) = \H_0k^{n+1} + \H_1k^{n}+\ldots,$$ for $\Q$-line bundles $\H_0$ and $\H_1$ (and where we use additive notation for tensor products). The identification $(\X_{\C\llparenthesis t\rrparenthesis},\L_{\C\llparenthesis t\rrparenthesis})\cong (X_{\C\llparenthesis t\rrparenthesis}, rL_{\llparenthesis t\rrparenthesis})$ gives a choice of trivialisations of these bundles away from $0 \in \Spec\C\llbracket t\rrbracket$, allowing us to make sense of $$b_0 = \deg \H_0 \textrm{ and } b_1=\deg\H_1,$$ by taking the orders of the resulting rational functions at the origin. Alternatively, the identification allows us to canonically produce a compactified family $\left(\bar{\X},\bar\L\right) \to \pr^1$ extending the original one and we may apply the Knudsen--Mumford expansion to the resulting family and take the (usual) degrees of the resulting $\Q$-line bundles on $\pr^1$; see \cite[Section 2.2]{DR} for more details on both approaches. Denote in addition the Hilbert polynomial of $(X,L)$ by $$h^0(X,kL) = a_0 k^n + a_1 k^{n-1}+\ldots.$$

\begin{definition}\label{SD2} 
The \emph{Donaldson--Futaki invariant} of $(\X,\L)$ is the quantity \begin{equation}\label{DF}\DF(\X,\L) = \frac{b_0a_1-b_1a_0}{a_0^2}.\end{equation}
\end{definition}

The approach taken here to the Donaldson--Futaki invariant is closely analogous to the work of Paul--Tian \cite{PT}, who consider test configurations rather than arcs. When $\X$  (hence $X$) is normal, this quantity is computed as an intersection number $$(L^n/2)\DF(\X,\L) = \frac{n}{n+1}\left(\frac{-K_X.L^{n-1}}{rL^n}\right)\L^{n+1}+\L^n.K_{\X/\Spec\C\llbracket t\rrbracket},$$ where the intersection numbers are meant in the sense of \cite[Section 2]{DR}, or may again equivalently be computed on the compactification $\left(\bar{\X},\bar\L\right)$. We define, for general $(\X,\L)$, the \emph{norm} of $(\X,\L)$ by $$\|(\X,\L)\| = \frac{\L.L^n}{(n+1)L^n} - \frac{\L^{n+1}}{(n+1)r^nL^n}\geq 0,$$ where we (by passing to a resolution of indeterminacy) may assume $\X$ admits a morphism to $X\times \Spec\C\llbracket t\rrbracket$ compatible with the identification in order to pull $L$ back to $\X$; this can also be defined by a formula analogous to Equation \eqref{DF} (see \cite[Definition 2.5]{uniform} for an analogous explanation for test configurations, which applies more generally to models).

\begin{definition}\cite{SD2,DR} We say that $(X,L)$ is 
\begin{enumerate}[(i)]
\item \emph{arc K-semistable} if for all models $(\X,\L)$ we have $\DF(\X,\L) \geq 0$;
\item \emph{uniformly arc K-stable} if there exists an $\epsilon>0$ such that for all models $(\X,\L)$ we have $\DF(\X,\L) \geq \epsilon \|(\X,\L)\|.$
\end{enumerate}
We say that $(X,L)$ is \emph{arc K-semistable at exponent} $r$ and  \emph{uniformly arc K-stable at exponent} $r$ respectively if the corresponding inequality holds for all models of exponent $r$.
\end{definition}

These imply the usual notions of K-semistability and uniform K-stability (namely, with respect to test configurations \cite{uniform,BHJ}). When $X$ is smooth, they are equivalent \cite{DR} to Tian's (analytic) notions of CM semistability and stability \cite{GT1}.  

\subsection{Stability of pairs}\label{sec:pairs} We recall the theory of stability of pairs, due to Paul \cite{SP}. Consider a linear algebraic group $G$ and a pair of $G$-representations $V,W$. We do not assume $G$ is reductive, though it will be so in applications. Fix $0\neq v\in V$ and $w\in W$.

\begin{definition}The pair $[v:w] \in \pr(V\oplus W)$ is \emph{semistable} if $$\overline{G.[v:w]}\cap \pr(0\oplus W) = \varnothing.$$\end{definition}

In the special case $W=\C$ with the trivial $G$-action and $w=1$, this definition recovers semistability of $v$ in classical geometric invariant theory. A main result of \cite{DR} is a numerical criterion for semistability of pairs, through arcs.

\begin{definition}
An \emph{arc} in $G$ is a morphism $$\rho: \Spec\C\llparenthesis t\rrparenthesis\to G,$$  namely a $\C\llparenthesis t\rrparenthesis$-point of $G$. We say that two arcs $\rho, \rho'$ are \emph{equivalent} if $\rho^{-1} \circ \rho'$ and $\rho' \circ \rho^{-1}$ are $\C{\llbracket t\rrbracket}$-points of $G$.
\end{definition}

\begin{remark}\label{rmk:arcs}If $X\hookrightarrow \pr(H^0(X,rL)) =\pr^{N_r}$, then equivalence classes of arcs in $\GL(N_r+1)$ are in bijection with models of exponent $r$ up to isomorphism \cite{SD2}\cite[Proposition 2.6]{DR}.\end{remark}

We next define numerical invariants. Choosing bases of $V$ and $W$, an arc $\rho$ induces vectors $\rho.v$ and $\rho.w$ with entries in $\C\llparenthesis t\rrparenthesis$, and we may define the order $\ord_0 \rho.v$ as the minimum of the orders of pole of the not-identically-zero entries of $\rho.v$, and likewise for $\rho.w$. It is straightforward to see these are independent of choice of bases.

\begin{definition}\cite{DR}
We define the \emph{weight} of $\rho$ to be $$\mu(\rho,[v:w]) = \ord_0\rho.w - \ord_0\rho.v,$$ and say that $[v:w]$ is \emph{numerically semistable} if $\mu(\rho,[v:w])\geq 0$ for all arcs $\rho$ in $G$.
\end{definition}

This again generalises the usual weight of a one-parameter subgroup in geometric invariant theory, where a one-parameter subgroup induces an arc in the obvious manner. We will make use of the following.

\begin{theorem}\label{numerical}\cite[Theorem 3.10]{DR}
A pair $[v:w]$ is semistable if and only if it is numerically semistable.
\end{theorem}

This result is optimal in general: Paul--Sun--Zhang show that one cannot detect semistability of pairs exclusively through one-parameter subgroups (their examples take $G=\SL(2)$) \cite{PSZ}, so one must incorporate arcs. 

Paul has also introduced a notion of a stable pair \cite{SP}, which we next recall. We fix an embedding $G\subset \GL(m)\subset \C^{m\times m}$ for some fixed $m$, and denote by $e \in \C^{m^2}$ the element corresponding to the identity matrix. Denote by $\deg V$ the smallest positive integer $d$ such that the weight polytope of  all nonzero $v \in V$ lies inside the rescaled unit simplex $d\Delta_{\dim V}$ \cite[(2.6)]{SP}.  For each $l>0$, we thus obtain a new pair $$\left[e^{\otimes \deg V} \otimes v^{\otimes l}: w^{\otimes l+1}\right]\in \pr\left(\left(\C^{m\otimes m}\right)^{\otimes d} \otimes V^{\otimes l} \oplus W^{\otimes l+1}\right).$$ Relatedly, we define the \emph{norm} of $\rho$ to be $$\|(\rho, [v])\| = \mu\left(\rho, \left[v:e^{\otimes \deg V}\right]\right) \geq 0,$$  where non-negativity can be obtained from \cite[Lemma 3.2]{GT3}.

\begin{definition}\label{numerical-stab-def}We define $[v:w]$ to be \emph{stable} if there is a $l>0$ such that $\big[e^{\otimes \deg V} \otimes v^{\otimes l}: w^{\otimes l+1}\big]$ is semistable, and say that $[v:w]$ is \emph{numerically stable} if there exists a $l>0$ such that for all arcs $\rho$ we have $$\mu(\rho, [v:w]) \geq (l+1)^{-1}\|(\rho, [v])\|.$$ \end{definition}

As with semistability, the two conditions agree.

\begin{theorem}\label{numerical-stable}\cite[Theorem 3.22]{DR}
A pair $[v:w]$ is stable if and only if it is numerically stable.
\end{theorem}

Further, the $l$s involved in stability and numerical stability can be taken to agree.

\section{Main results}
\subsection{Constructibility of semistability of pairs} 

We consider a general projective variety $Y$ with an action of an algebraic group $G$, and let $\hat W\subset Y$ be a closed subvariety. For $y\in Y$, we consider the condition that $$\overline{G.y} \cap \hat W \neq \varnothing;$$ if this is the case, we say that $y$ \emph{degenerates} to $\hat W$ under $G$. Note that this is a condition on the orbit $G.y$ of $y$ rather than $y$ itself. 

\begin{proposition}\label{prop:main} The locus of points $y\in Y$ which do not degenerate to $\hat W$ is constructible in $Y$.
\end{proposition}

That is, the locus of points not degenerating to $\hat W$ is a finite union of locally closed subsets of $Y$, where we recall a locally closed subset is an open subset of a closed subset of $Y$.

\begin{proof}

By arguing on each irreducible component, we may assume $Y$ is irreducible.  Let $\overline G$ be a projective compactification of $G$, and consider $S\subset Y$ a locally closed subvariety. Consider the closure $\bar \Gamma_S$ of the graph $\Gamma_S$ of the morphism $$G\times S \to Y, \qquad (g,s)\to g(s)$$ induced by the $G$-action,  so that  $\bar \Gamma_S \subset \bar G \times S \times Y$. For $s\in S$, denote by $\bar\Gamma_{S,s}$ and $\Gamma_{S,s}$ the fibres over $s\in S$ defined via the natural projections. Thus the fibre $\Gamma_{S,s}$ is the graph of the morphism $G\to Y$ defined by $g \to g(s)$.

Our proof involves a stratification of $Y$. We first describe the open stratum, for which we involve $\Gamma_Y$ and $\bar\Gamma_Y$. Firstly, flattening stratification applied to the morphism $\overline{\Gamma}_Y \to Y$ (the projection onto the second factor of $\bar G\times Y \times Y$, so that the target $Y$ of $\overline{\Gamma}_Y \to Y$ plays the role of $S$) produces an open set $\tilde Y^0 \subset Y$ for which the restricted morphism $\overline{\Gamma}_{Y}|_{\tilde Y^0} \to \tilde Y^0$ is flat. Note that ${\Gamma}_{Y}|_{\tilde Y^0} = {\Gamma}_{\tilde Y^0}$ by definition  and hence $\overline{\Gamma}_{Y}|_{\tilde Y^0} = \overline{\Gamma}_{\tilde Y^0}$.

 We claim there is an open subvariety $Y^0\subset \tilde Y^0$ such that for all $y\in Y^0$, the fibre $\Gamma_{Y,y}$ is dense in $\overline{\Gamma}_{Y,y}$.  To prove the claim, let $Y^{0}$ denote the locus of $y \in \tilde Y^0$ for which $\Gamma_{Y,y}$ is dense in $\overline{\Gamma}_{Y,y}$, so that we aim to show $Y^{0}\subset \tilde Y^0$ is open.  Note that $\Gamma_{\tilde Y^0}\cong G\times \tilde Y^0$, being the graph of a morphism from $G\times \tilde Y^0$. Each connected component of $\bar \Gamma_{\tilde Y^0}$ thus has dimension $\dim \tilde Y^0+\dim G$ (where we note that $G$ may not be connected, hence the need to restrict to a connected component). Since $\bar \Gamma_{\tilde Y^0} \to \tilde Y^0$ is flat, it follows that each fibre is pure dimensional, in the sense that each irreducible component of $\bar \Gamma_{Y,y}$ has the same dimension, which by flatness must equal to $\dim G$.  By pure dimensionality,  $\Gamma_{Y,y}$ is dense in $\bar\Gamma_{Y,y}$ if and only if the fibre of the complement $\bar\Gamma_{Y,y}\setminus\Gamma_{Y,y}$ contains no irreducible component of  $\bar\Gamma_{Y,y}$, which is equivalent to $\dim \bar\Gamma_{Y,y}\setminus\Gamma_{Y,y}< \dim G$. Since $\bar\Gamma_{\tilde Y_0}\setminus \Gamma_{\tilde Y_0}$ is a closed subvariety of $\bar\Gamma_{\tilde Y_0}$ such that $$\bar\Gamma_{Y,y}\setminus\Gamma_{Y,y} = (\bar\Gamma_{\tilde Y_0}\setminus \Gamma_{\tilde Y_0})_y,$$ we see that the locus  in $Y^0$ for which $$\dim \bar\Gamma_{Y,y}\setminus\Gamma_{Y,y} <\dim G$$ is open, proving the claim. 

We next prove that, within $Y^0$, the locus of points which degenerate to $\hat W$ is closed. Indeed, set $$B_{Y_0} = \overline{\Gamma}_{Y^0} \cap  \bar G \times Y^0 \times \hat W,$$ which is a closed subvariety of $\overline{\Gamma}_{Y^0}$. Denoting by  $\pi_{Y_0}: \bar G \times Y^0 \times \hat W \to Y^0$ the projection, by fibrewise density, $y\in \pi_{Y_0}(B_{Y_0})$ if and only if $\overline{G.y} \cap \hat W \neq \varnothing$. Since $B_{Y_0}$ is closed,  its image  $\pi_{Y_0}(B_{Y_0})$  is also closed, as $\pi_{Y_0}$ is a projective morphism. We have thus proven that the locus of points in $Y_0$ which degenerate to $\hat W$ is closed in $Y_0$.

We next consider $Y \setminus Y_0$. This admits open sets $Y^1\subset \tilde Y^1 \subset Y \setminus Y^0$ satisfying the same conditions as $Y^0$, where these conditions are defined via  $\Gamma_{\tilde Y^1}\to \tilde Y^1$ and $\bar \Gamma_{\tilde Y^1}\to \tilde Y^1$. By Noetherianity, this process terminates in finitely many steps, producing a finite stratification $Y = \cup_i Y^i$, such that for all $Y^i$, the locus of points which do not degenerate to $\hat W$ is open in $Y^i$, hence is locally closed in $Y$. As such, the locus of points in $Y$ which do not degenerate to $\hat W$ is a finite union of locally closed subsets, hence is constructible.
\end{proof}

\begin{example}
It is not true that degenerating to $\hat W$ is a Zariski closed property, as one sees by taking $Y = \pr^1$, $W=[0,1]$ and endowing $Y$ with the $\C^*$-action $[x:y] \to [tx:y]$. The locus of points degenerating to $\hat W$ is then $\pr^1 \setminus [1:0]$.
\end{example}

We apply this to semistability of pairs, so consider pairs lying in $\pr(V\oplus W)$ with $G$ acting on $V$ and $W$ as in Section \ref{sec:pairs}.

\begin{corollary}\label{cor:zariski-pairs}
The locus of semistable pairs is constructible in $\pr(V\oplus W)$.
\end{corollary}

\begin{proof} 
This is immediate from Proposition \ref{prop:main}, by taking $Y$ to be $\pr(V\oplus W)$ and $\hat W$ to be $\pr(0\oplus W)$. The condition that $[v:w]$ is not semistable is precisely the condition that $[v:w]$ degenerates to $\hat W=\pr(0\oplus W)$ under $G$, and the complement of this  locus is constructible by Proposition \ref{prop:main}.
\end{proof}

We obtain the same result for stability of pairs, following the notation of Section \ref{sec:pairs}, essentially since this is defined to be semistability of an associated pair.

\begin{corollary}\label{cor:zariski-pairs-stable}
The locus of stable pairs is a countable union of constructible sets. 
\end{corollary}

\begin{proof}

Defining $[v:w]$ to be \emph{stable at level} $l$ if $\left[e^{\otimes \deg V} \otimes v^{\otimes l}: w^{\otimes l+1}\right]$ is semistable, then for $[v:w]$ to be stable means that it is stable at level $l$ for some $l>0$, and for each such $l$ the locus of pairs which are stable at level $l$ is constructible by Corollary \ref{cor:zariski-pairs}.  Thus the locus of stable pairs is a countable union of constructible sets.
\end{proof}

We note that the locus of points which are stable at level $l$ is increasing as $l\to \infty$: stability of level $l$ implies stability at level $l+l'$ for any $l'>0$, as follows from the numerical criterion and nonnegativity of the norm.

\subsection{Semistability of pairs as semistability without ampleness}\label{non-ample} Our application of Corollary \ref{cor:zariski-pairs-stable} will pass through the following construction. Let $Y$ be a projective scheme, and let $L$ be a line bundle, which we do not assume is ample. We let $G$ be a linear algebraic group acting on $(Y,L)$, in such a way that the $G$-action lifts to some auxiliary ample line bundle on $Y$. We may thus write additively $$L = A - B,$$ where $A,B$ are $G$-linearised very ample line bundles and this is an equality of $G$-linearised line bundles. Setting $V = H^0(Y,A)$ and $W=H^0(Y,B)$, this produces a pair of embeddings via a choice of bases $$Y \hookrightarrow \pr(V), \qquad Y \hookrightarrow \pr(W).$$ A point $y \in Y$  thus induces a  pair of points $[v]\in\pr(V)$ and $[w]\in\pr(W)$, and we may choose lifts $v\in V, w\in W$, giving $[v:w] \in \pr(V\oplus W)$. Notice that, in contrast to an arbitrary pair, $w\neq 0$.

The point $[v:w]$ is not the only point in $\pr(V\oplus W)$ corresponding to $y$, because of the ambiguity in choosing the lifts $v$ and $w$ of $[v]$ and $[w]$ respectively. The space of lifts correspond to the orbit of the natural $\C^*$-action on $\pr(V\oplus W)$, scaling $[v:w]\to [tv:w]$: for $t\in \C^*$.

\begin{lemma}\label{scaling}
For $t\in \C^*$, a pair $[v:w]$ is semistable if and only if $[tv:w]$ is semistable.
\end{lemma}

\begin{proof}
This is a simple consequence of the numerical criterion for semistability of pairs, namely Theorem \ref{numerical}.
\end{proof}

Consider the morphism \begin{equation}\label{fibre-bundle}\phi: \pr(V\oplus W)\setminus(\pr(V\oplus 0) \cup \pr(0\oplus W)) \to \pr(V)\times \pr(W)\end{equation} defined by $\phi([v:w]) = [v:w]$. Its fibre over $([v],[w])\in \pr(V)\oplus \pr(W)$ is the $\C^*$-orbit $\C^*. [v:w]$. Over a standard affine chart $U_i\times U_j$ of $ \pr(V)\times \pr(W)$, the morphism reduces to the trivial projection $\C^{\dim V-1} \times \C^{\dim W-1}\times\C^* \to \C^{\dim V-1} \times \C^{\dim W-1}$, so the morphism is a (Zariski) locally trivial $\C^*$-bundle, and one may thus view $\phi$ as a geometric quotient under the $\C^*$-action.

Our aim is to compare semistability, as defined for the $G$-action on $(X,H)$, to that in $\pr(V\oplus W)$. To make sense of this, as $H$ may not be ample, we give the following definition.

\begin{definition}\label{H-semistable}
We say that $y\in Y$ is \emph{$H$-semistable} if for all arcs $\rho$ in $G$ we have $\mu(\rho,y)\geq 0$, where $$\mu(\rho,y) = \mu(\rho,w) - \mu(\rho,v).$$
\end{definition}

That is, the weight with respect to $H$ is the difference of the weights with respect to $A$ and $B$, where as $A$ and $B$ are ample, these agree with the quantities defined originally defined by Donaldson. 

\begin{lemma}\label{additivity}
The quantity $\mu(\rho,y)$ is independent of choice of $A, B$ for which $H=A-B$.
\end{lemma} 

\begin{proof}

We  give a more intrinsic approach to the weight, which is similar to  \cite[Lemma 3.7]{DR}.  Choosing a point $y\in Y$, the arc $\rho$ induces morphisms $\Spec\C{\llbracket t\rrbracket}\to Y$, by the valuative criterion for properness, and we pull back the line bundle $A$ to $\Spec\C{\llbracket t\rrbracket}$. Through the identification of $(\Spec\C{\llbracket t\rrbracket}, A)$ over $\Spec\C\llparenthesis t\rrparenthesis$  with $\rho.[v] \cong \Spec\C\llparenthesis t\rrparenthesis$ induced by the arc, a section $s_A\in A_y$ induces a section over $\rho.[v]$ and hence  a rational section $\tilde s_A$ of $A$. If we now assume that $Y$ is embedded into $\pr(V)$ for $V=H^0(Y,A)$ and $y=[v]$, then the weight  satisfies $$\ord_0\rho.v = \inf\{m\in \Z: t^m\tilde s\textrm{ is regular}\}.$$ 

Given now linearised very ample line bundles $A_1, A_2$, we may choose the tensor product section $ s_{A_1}\otimes s_{A_2}$, which induces the rational section $\tilde  s_{A_1}\otimes \tilde s_{A_2}$ of $A_1\otimes A_2$. As such, the weight is additive under tensor products, which implies the result.\end{proof}

We may thus view $H$-semistability as semistability of an associated pair. The following is thus a simple consequence of the numerical criterion.

\begin{lemma}\label{semistability-vs-pairs}
The point  $y\in Y$ is  $H$-semistable if and only if $[v:w]$ is semistable.
\end{lemma}

We next prove a version of constructibility for $H$-semistability, which does not require ampleness of $H$.

\begin{corollary}\label{H-ss-zar}
The $H$-semistable locus in $Y$ is constructible. 
\end{corollary}

\begin{proof}
Corollary \ref{cor:zariski-pairs} implies that the locus of semistable pairs is constructible in $\pr(V\oplus W)$; denote this subset as $\pr(V\oplus W)^{\ses}$. Consider also the  morphism $\phi: \pr(V\oplus W)\setminus(\pr(V\oplus 0) \cup \pr(0\oplus W)) \to \pr(V)\times \pr(W)$ defined in Equation \eqref{fibre-bundle}. As $\phi$ is a geometric quotient and $\pr(V\oplus W)^{\ses}$ is $\C^*$-invariant by Lemma \ref{scaling}, there is a well-defined constructible locus $(\pr(V)\times \pr(W))^{\ses}$ defined by $$(\pr(V)\times \pr(W))^{\ses} = \phi((\pr(V\oplus W)\setminus(\pr(V\oplus 0) \cup \pr(0\oplus W)))^{\ses}),$$ where constructibility follows from Chevalley's theorem on constructible sets.

By Lemma \ref{semistability-vs-pairs}, $y\in Y$ is $H$-semistable if and only if $(y,y)\in (Y\times Y)\cap (\pr(V)\times \pr(W))^{\ses}$. Said another way, the $H$-semistable locus in $Y$ is the intersection $\Delta \cap (\pr(V)\times \pr(W))^{\ses}$, where $\Delta$ denotes the diagonal, through the diagonal embedding of $Y$ into $\pr(V)\times \pr(W)$. It follows that the $H$-semistable locus  is constructible in Y. \end{proof}

The same ideas apply to prove that the $H$-stable locus in $Y$ is a countable union of constructible sets, defining $H$-stability again through the numerical criterion analogously to Definitions \ref{numerical-stab-def} and \ref{H-semistable}. We sketch the details. Recall that stability of a pair $[v:w]$ corresponds to the existence of an $l>0$ such that   $$\left[e^{\otimes \deg V} \otimes v^{\otimes l}: w^{\otimes l+1}\right]\in \pr\left(\left(\C^{m\otimes m}\right)^{\otimes d} \otimes V^{\otimes l} \oplus W^{\otimes l+1}\right)$$ is semistable. Fixing such an $l$, writing $H=A-B$ as before, we obtain a pair of embeddings $Y\hookrightarrow \pr((\C^{m\otimes m})^{\otimes d} \otimes V)$ defined by $y\to [e^{\otimes d}\otimes v^{\otimes l}]$. If one chooses a basis of $V = H^0(Y,A)$, then this morphism corresponds to the composition of the embeddings  $Y\hookrightarrow\pr(e^{\otimes d}\otimes V^{\otimes l})\hookrightarrow \pr((\C^{m\otimes m})^{\otimes d} \otimes V^{\otimes l})$. We obtain also a second embedding  $Y\hookrightarrow \pr(W^{\otimes{l+1}})$, and arguing just as before, the locus of points for which $y\in Y$ satisfies the numerical criterion $$\mu(\rho,y) \geq (l+1)^{-1}\|(\rho,y)\|$$ for all arcs $\rho$ is constructible. Taking the union over all $l\geq 0$ proves the following.

\begin{corollary}\label{H-s-zar}
The $H$-stable locus in $Y$ is a countable union of constructible sets. 
\end{corollary}

\subsection{Applications to uniform arc K-stability}

We fix an exponent $r$ and prove that arc K-semistability at exponent $r$ is a constructible property. We thus embed $$X \hookrightarrow \pr(H^0(X,rL)) = \pr^{N_r},$$ giving a point $[X] \in \Hilb_{h(k)}\left(\pr^{N_r}\right)$ where $h(k)$ is the Hilbert polynomial of $(X,rL)$. The construction of the Hilbert scheme endows it with a sequence of line bundles $\L_{\Hilb, j}$ arising from its construction as a subscheme of a $j$-dependent projective space for $j \gg 0$. In more detail, the Hilbert scheme is constructed as a subscheme of a $j$-dependent Grassmannian, which we embed into projective space via the Pl\"ucker embedding. We restrict the resulting $\scO(1)$ line bundle from these projective spaces to $\Hilb_{h(k)}\left(\pr^{N_r}\right)$ to obtain the very ample line bundles $\L_{\Hilb, j}$. 

The Hilbert scheme also admits a universal family $$\pi_{\U}: (\U,\L_{\U}) \to \Hilb_{h(k)}\left(\pr^{N_r}\right)$$ with $\L_{\U}$ relatively ample, and so we may apply the Knudsen--Mumford theory \cite[Theorem 4]{KM} to the pushforwards $\pi_{\U*}(k\L_{\U})$ to obtain a sequence of line bundles $\H_m$ on $\Hilb_{h(k)}\left(\pr^{N_r}\right)$ arising from an expansion of the line bundles $\det\pi_{\U*}(k\L_{\U})$. 

We incorporate the $\GL(N_r+1)$-action on $H^0(X,rL)$, which induces an action on  $\Hilb_{h(k)}\left(\pr^{N_r}\right)$ lifting to  $\L_{\Hilb, j}$ by construction. Similarly the action lifts to $\L_{\U}$ and hence each $\H_m$. We may thus write the difference $$\L_{\CM} := \frac{\H_0a_1 - \H_1a_0}{a_0^2} = L_0 - L_1$$ as a difference of $\GL(N_r+1)$-linearised very ample line bundles (perhaps after scaling to clear denominators), using that the action lifts to the very ample line bundles $\L_{\Hilb, j}$. We note here that $\L_{\CM}$ frequently fails to be ample (and can even be negative on curves) \cite{FR}, necessitating the theory of stability of pairs.

We thus have a pair of embeddings $$\Hilb_{h(k)}\left(\pr^{N_r}\right) \hookrightarrow \pr\left(H^0\left(\Hilb_{h(k)}\left(\pr^{N_r}), L_m\right)\right)\right)$$ for $m=0,1$ with $\GL(N_r+1)$-actions on $$V = H^0\left(\Hilb_{h(k)}\left(\pr^{N_r}\right), L_0\right), \qquad W =  H^0\left(\Hilb_{h(k)}\left(\pr^{N_r}\right), L_1\right).$$ We are thus precisely in the situation of Section \ref{non-ample}, meaning that $\L_{\CM}$-stability of $[X]\in \Hilb_{h(k)}(\pr^{N_r})$ corresponds precisely to stability of the associated pair $[v_X:v_w] \in \pr(V\oplus W)$ for any lifts $v_X\in V, w_X\in W$ by Lemma \ref{semistability-vs-pairs}.

We next fix a flat family $\pi_{\X}: (\X,\L) \to B$ of polarised varieties, so $\L$ is relatively ample. We assume $r\L$ is relatively very ample and embeds $\X \to \pr(\pi_{\X*}(r\L))$ into a projective bundle, which is automatic for $r \gg 0$. For each $b\in B$ we may thus ask whether $(\X_b,\L_b)$ is arc K-semistable at exponent $r$. 

\begin{corollary}\label{fixed-exponent}

The arc K-semistable locus is a countable intersection of constructible sets. The uniformly arc K-stable locus is a countable union of a countable intersection of constructible sets.

\end{corollary}

\begin{proof}

We cover $B$ by finitely many affine subvarieties trivialising the bundle $\pi_{\X*}(r\L)$. Fix one such $U\subset B$, embedding $\X_U \hookrightarrow U \times \pr^{N_r}$. From the discussion preceding the statement of the present corollary, we obtain an $\L_{\CM}$-stability problem for each $[\X_b]\in \Hilb_{h(k)}(\pr^{N_r})$, and by Corollaries  \ref{H-ss-zar} and \ref{H-s-zar}, the $\L_{\CM}$-semistable and $\L_{\CM}$-stable loci are constructible  and countable unions of constructible sets in $U$ respectively. 

To relate this to arc K-semistability and uniform arc K-stability at exponent $r$, we appeal to the numerical criteria defining  $\L_{\CM}$-semistability and stability. For each arc $\rho$ in $\GL(N_r+1)$, we thus wish to compute $\mu(\rho,[X])$. The arc $\rho$ produces a $\C\llparenthesis t\rrparenthesis$-point of $\Hilb_{h(k)}\left(\pr^{N_r}\right)$, hence a unique $\C{\llbracket t\rrbracket}$-point by the valuative criteria for properness and separatedness applied to the Hilbert scheme. The weight $\mu(\rho,[v_X,w_X]) $ is then the degree of $\L_{\CM}$ restricted to this subscheme by \cite[Lemma 3.7]{DR} (whose proof is similar to Lemma \ref{additivity}), where again the degree is meant in the  same sense as Section \ref{sec:arcs-prelims} using the natural trivialisation of the arc.   The arc $\rho$ also induces a model $(\X,\L)$ through from Remark \ref{rmk:arcs} (namely, by restricting the universal family over the Hilbert scheme) in such a way that $$\DF(\X,\L) = \mu(\rho,[v_X,w_X]),$$ since $\DF(\X,\L)$ is also defined through degrees of line bundles associated to the Knudsen--Mumford expansion, and where we use functoriality of this expansion. 

To understand the norm, we argue in a different manner and appeal to \cite[Corollary 4.7 (iii)]{DR}; while \cite{DR} assumes normality of $X$, the proof applies in general (or can be reduced to the normal case by normalising).

Thus $\L_{\CM}$-semistability and stability precisely correspond to arc K-semistability and uniform arc K-stability at exponent $r$. It follows that arc K-semistability is a countable intersection of constructible sets. By definition, uniform arc K-stability asks for the existence of an $l>0$ such that for all models $(\X,\L)$ we have $$\DF(\X,\L)\geq \frac{1}{l+1}\|(\X,\L)\|.$$ For each fixed $l$, the corresponding locus is a countable intersection of constructible sets, and so the uniformly arc K-stable locus (being a union over all such $l$) is a countable union of countable intersections of constructible sets. Since this is true for each $U$, we obtain the statement. \end{proof}

We may now apply our results to cscK metrics, proving Theorem \ref{intromainthm}. Consider a flat family $(\X,\L)\to B$ of smooth polarised varieties such that each fibre has finite automorphism group. 

\begin{corollary}\label{cor:final}
The cscK locus $B_{\cscK}$ is very general in $B$. \end{corollary}

\begin{proof}
By Trusiani \cite{AT}, a fibre $(\X_b,\L_b)$ is cscK if and only if it is uniformly K-stable. Since the cscK condition implies uniform arc K-stability \cite{DR}, which in turn implies uniform K-stability, it follows that uniform arc K-stability is equivalent to the existence of a cscK metric. Now, by \cite{FS} the cscK locus is open in the analytic topology, while by Corollary \ref{fixed-exponent} the uniformly arc K-stable locus is a countable union of countable intersections of constructible sets. It follows from Proposition \ref{topology} that a set which is both open in the analytic topology and which is a countable union of countable intersections of constructible sets is in fact very general, and so the proof is complete. \end{proof}

We used the following result, which is likely well known (see \cite[XI Corollary 2.3]{GR} for the closely analogous statement that a constructible set which is analytically open is Zariski open).

\begin{proposition}\label{topology}
Let $B^{\circ}\subset B$ be open in the analytic topology, and equal to a countable union of countable intersections of constructible sets. Then $B^{\circ}$ is very general. 
\end{proposition}

We recall we allow $B^{\circ}$ to be empty. The main point is that an analytically open set cannot be contained in a countable union of Zariski closed sets, as the latter has zero measure. This implies the claim when $\dim B=1$ in a straightforward manner, and in higher dimensions we show that $B\setminus B^{\circ}$ must be contained in a countable union of Zariski closed subsets; an induction on dimension strategy then proves that $B\setminus B^{\circ}$ is actually equal to a countable union of Zariski closed subsets. We may assume $B$ is irreducible and $B^{\circ}$ is nonempty, as in the former case we may restrict to an irreducible component, while in the latter the result is trivial.

In order to prove Proposition \ref{topology}, we first prove the following weaker result, for which we may again assume $B$ is irreducible.

\begin{lemma}\label{topology2}
The complement $B\setminus B^{\circ}$ is contained in a countable union of Zariski closed proper subsets of $B$.
\end{lemma}

Here we mean proper in the set-theoretic sense, namely not equal to $B$ itself.

\begin{proof}

Write $B^{\circ} = \cup_{i}(\cap_{j} B_{i,j})$, where the $B_{i,j}$ are constructible. Let $C_i =\cap_{j\geq 1} B_{i,j}$, so that $B^{\circ}  = \cup_i C_i$. 

We first claim that either $C_i$ is contained in a proper Zariski closed subset of $B$, or $B \setminus C_i$ is contained in a countable union of Zariski closed subsets of $B$. To prove the claim, note that either every $B_{i,j}$ is dense in the Zariski topology of $B$, or some such $B_{i,j}$ is not Zariski dense. If $B_{i,j}$ is not Zariski dense, then $C_i \subset B_{i,j}\subset \overline{B_{i,j}}$ (this denoting the Zariski closure), so $C_i$ is contained in a proper Zariski closed subset of $X$. Otherwise, if every $B_{i,j}$ is Zariski dense, each $B_{i,j}$ contains a (nonempty) Zariski open set $U_{i,j}$, being constructible. Then $C_i  =\cap_{j\geq 1} B_{i,j} \supset \cap_j U_{i,j}$ contains a countable intersection of Zariski open subsets. Thus $B\setminus C_i \subset \cup_j (B\setminus U_{i,j})$ is contained in a countable union of Zariski closed subsets of $B$.

Since the union $B^{\circ} =  \cup_{i\geq 1}C_i$ is open in the analytic topology, it follows that  $B \setminus C_i$ is contained in a countable union of Zariski closed subsets of $B$ for some $i$. Indeed, otherwise each $C_i$ would be contained in a Zariski closed subset of $B$ by the claim just proven, and hence would have measure zero (after fixing a measure induced by the volume form associated to a K\"ahler metric, for example). As such, $B^{\circ}$ would have measure zero, contradicting openness in the analytic topology.   Thus $B\setminus B^{\circ} \subset B \setminus C_i$ is contained in a countable union of Zariski closed subsets of $B$, proving the result. \end{proof}

We may now prove Proposition \ref{topology}.

\begin{proof}[Proof of Proposition \ref{topology}] We prove the result by induction on $\dim B$, the case when $\dim B=0$ being trivial. By Lemma \ref{topology2} we may choose Zariski closed subsets $V_i$ of $B$ such that  $(B\setminus B^{\circ}) \subset \cup_i V_i$, where we may in addition assume the $V_i$ are irreducible. 

Fixing a single $V_i$, the set $(B^{\circ}\cap V_i)\subset V_i$ is analytically open, and is given by a countable union of countable intersections of constructible sets in $V_i$. As such, by induction $V_i \setminus (B^{\circ}\cap V_i)$ is a countable union of proper Zariski closed subsets of $V_i$, or $B^{\circ}\cap V_i = \varnothing$. In either case, since Zariski closed subsets of $V_i$ are Zariski closed subsets of $B$, it follows that $V_i\setminus(B^{\circ} \cap V_i)$ is equal to a countable union of proper Zariski closed subsets of $B$. Since $(B\setminus B^{\circ}) \subset \cup_i V_i$, the result follows. \end{proof}


\vspace{4mm}

\begin{thebibliography}{99}

\bibitem{ADVS}
C. Arezzo, A. Della Vedova and Y. Shi, \emph{Constant scalar curvature {K}\"ahler metrics on ramified  {G}alois coverings}. J. Reine Angew. Math. 799 (2023), 229--247.

\bibitem{calabi}
C. Araujo, A.-M. Castravet, I. Cheltsov, K. Fujita, A.-S. Kalaghiros, J. Martinez-Garcia, C. Shramob, H. Suess, N. Viswanathan. \emph{The Calabi problem for Fano threefolds.} London Math. Soc. Lecture Note Ser., 485
Cambridge University Press, Cambridge, 2023. vii+441 pp.
\bibitem{BL}
H. Blum and Y. Liu, \emph{The normalized volume of a singularity is lower semicontinuous}.  J. Eur. Math. Soc. (JEMS) 23 (2021), no. 4, 1225--1256.

\bibitem{BL2}
H. Blum and Y. Liu, \emph{Openness of uniform {K}-stability in families of {$\Q$}-{F}ano varieties}. Ann. Sci. \'Ec. Norm. Sup\'er. (4) 55 (2022), no. 1, 1--41.

\bibitem{BLX}
H. Blum, Y. Liu and C. Xu, \emph{Openness of K-semistability for {Fano} varieties}. Duke Math. J. 171 (2022), no. 13, 2753--2797.


\bibitem{BX}
H. Blum and C. Xu, \emph{Uniqueness of K-polystable degenerations of Fano varieties}. Ann. of Math. (2) 190 (2019), no. 2, 609--656.

\bibitem{BHJ}
S. Boucksom, T. Hisamoto and M. Jonsson, \emph{Uniform {K}-stability, {D}uistermaat-{H}eckman measures and
              singularities of pairs}. Ann. Inst. Fourier (Grenoble) 67 (2017), no. 2, 743--841.

\bibitem{BJ}
S. Boucksom and M. Jonsson, \emph{On the Yau-Tian-Donaldson conjecture for weighted cscK metrics}. arXiv:2509.15016, 84pp.

\bibitem{chencheng:iiiexistence}
X. Chen and J. Cheng, \emph{On the constant scalar curvature {K}\"{a}hler metrics
              ({II})---{E}xistence results}. J. Amer. Math. Soc. 34 (2021), no. 4, 909--936.


\bibitem{CP}
G. Codogni and Zs. Patakfalvi, \emph{Positivity of the {CM} line bundle for families of {$K$}-stable klt {F}ano varieties}. Invent. Math. 223 (2021), no. 3, 811--894.

\bibitem{CVS}
S. Cynk and D. van Straten, \emph{Infinitesimal deformations of double covers of smooth algebraic varieties}. Math. Nachr. 279 (2006), no. 7, 716--726.

\bibitem{DZ}
T. Darvas and K. Zhang, \emph{A YTD correspondence for constant scalar curvature metrics}. arXiv:2509.15173, 52pp.

\bibitem{RD-alpha}
R. Dervan, \emph{Alpha invariants and {K}-stability for general polarizations
              of {F}ano varieties}. Int. Math. Res. Not. IMRN 2015, no. 16, 7162--7189.

\bibitem{uniform}  
R. Dervan, \emph{Uniform stability of twisted constant scalar curvature Kähler metric} Int. Math. Res. Not. IMRN 2016, no. 15, 4728--4783.    

\bibitem{RD:mabuchi}
  R. Dervan, \emph{         Alpha invariants and coercivity of the Mabuchi functional on Fano manifolds.} Ann. Fac. Sci. Toulouse Math. (6) 25 (2016), no. 4, 919--934.
                         
\bibitem{DN}
R. Dervan and P. Naumann, \emph{Moduli of polarised manifolds via canonical K{\"a}hler metrics}. arXiv:1810.02576, to appear in \emph{Ann. Inst. Fourier}, 28pp.

\bibitem{DR}
R. Dervan and R. Reboulet, \emph{Arcs, stability of pairs and the Mabuchi functional}. arXiv:2409.13617, 38pp.

\bibitem{SD1}
S. K. Donaldson, \emph{Scalar curvature and projective embeddings. I} J. Differential Geom. 59 (2001), no. 3, 479--522. 

\bibitem{SD2}
S. K. Donaldson, \emph{Stability, birational transformations and the K\"ahler-Einstein problem.} Surveys in differential geometry. Vol. XVII, 203–228. 

\bibitem{SD3}
S. K. Donaldson, \emph{Algebraic families of constant scalar curvature {K}\"ahler   metrics}. Surveys in differential geometry 2014. Regularity and evolution of nonlinear equations, 111--137. Surv. Differ. Geom., 19.

\bibitem{donaldson:toric}
S. K. Donaldson, \emph{Constant scalar curvature metrics on toric surfaces.} Geom. Funct. Anal. 19 (2009), no. 1, 83--136.

\bibitem{FR}
J. Fine and J. Ross, \emph{A note on positivity of the {CM} line bundle}. Int. Math. Res. Not. 2006, Art. ID 95875, 14 pp.

\bibitem{FS}
A. Fujiki and G. Schumacher, \emph{The moduli space of extremal compact Kähler manifolds and generalized Weil-Petersson metrics} Publ. Res. Inst. Math. Sci. 26 (1990), no. 1, 101--183.

\bibitem{GR}
A. Grothendieck and M. Raynaud, \emph{Revêtements étales et groupe fondamental (SGA 1)}. Doc. Math. (Paris), 3. Société Mathématique de France, Paris, 2003. xviii+327 pp.

\bibitem{IS}
N. Ilten and H. Suess, \emph{K-stability for Fano manifolds with torus action of complexity 1.} Duke Math. J. 166 (2017), no. 1, 177--204.


\bibitem{KM}
F. Knudsen and D. Mumford, \emph{The projectivity of the moduli space of stable curves {I}: preliminaries on det and {D}iv}. Math. Scand. 39 (1976), no. 1, 19--55.

\bibitem{CL}
C. Li, \emph{Geodesic rays and stability in the cscK problem}. Ann. Sci. Éc. Norm. Supér. (4) 55 (2022), no. 6, 1529--1574.



\bibitem{LWX}
C. Li, X. Wang and C. Xu, \emph{On the proper moduli spaces of smoothable K\"ahler-Einstein Fano varieties.} Duke Math. J. 168 (2019), no. 8, 1387--1459.

\bibitem{LXZ}
Y. Liu, X. Wang, C. Xu, \emph{Finite generation for valuations computing stability thresholds and applications to K-stability}. Ann. of Math. (2) 196 (2022), no. 2, 507--566.


\bibitem{GIT}
D. Mumford, J. Fogarty and F. Kirwan, \emph{Geometric invariant theory.} Third edition. Ergeb. Math. Grenzgeb. (2), 34, Springer-Verlag, Berlin, 1994. xiv+292 pp.

\bibitem{YO}
Y. Odaka, \emph{On the moduli of {K}\"ahler--{E}instein {F}ano manifolds}. S\"ugaku 72 (2020), no. 4, 337--364.


\bibitem{AO}
A. Ortu, \emph{Moment maps and stability of holomorphic submersions}. arXiv:2407.03246, 28pp.

\bibitem{SP}
S. T. Paul, \emph{Hyperdiscriminant polytopes, Chow polytopes, and Mabuchi energy asymptotics.} Ann. of Math. (2) 175 (2012), no. 1, 255--296.

\bibitem{SP2}
S. T. Paul, \emph{CM Stability of Projective varieties}. arXiv:1206.4923, 52pp.


\bibitem{PSZ}
S. T. Paul, S. Sun and J. Zhang, \emph{On the generalized numerical criterion}. Int. Math. Res. Not. IMRN 2024, no. 17, 12150--12160.

\bibitem{PT}
S. T. Paul and G. Tian, \emph{CM Stability and the Generalized Futaki Invariant I}. arXiv:math/0605278, 24pp. 

\bibitem{RR}
R. Reboulet, \emph{A Hilbert--Mumford criterion for generalised Monge--Amp{\`e}re equations}. To appear in Ann. Sc. Norm. Super. Pisa Cl. Sci., 22pp.

\bibitem{GT1}
G. Tian, \emph{K\"ahler-Einstein metrics with positive scalar curvature}. Invent. Math. 130 (1997), no. 1, 1--37.

\bibitem{GT2}
G. Tian, \emph{Stability of pairs}. arXiv:1310.5544, 15pp.


\bibitem{GT3}
G. Tian, \emph{K-stability implies {CM}-stability}. Geometry, analysis and probability, 245--261. Progr. Math., 310
Birkhäuser/Springer, Cham, 2017.

\bibitem{AT}
A. Trusiani, \emph{A solution to the Yau-Tian-Donaldson Conjecture through Special Fujita Approximations}. arXiv:2605.30063, 28pp.


\bibitem{XW}
X. Wang, \emph{Height and GIT weight}. Math. Res. Lett. 19 (2012), no. 4, 909--926.
\bibitem{wang-zhou-toric}

X.-J. Wang, B. Zhou, \emph{On the existence and nonexistence of extremal metrics on toric
              {K}\"ahler surfaces}. Adv. Math. 226 (2011), no. 5, 4429--4455.


\bibitem{QW}
Q. Westrich, \emph{Discriminants and Higher K-energies on Polarized K{\"a}hler Manifolds}. arXiv:1507.01152, 15pp.

\bibitem{CX}
C. Xu, \emph{A minimizing valuation is quasi-monomial}. Ann. of Math. (2) 191 (2020), no. 3, 1003--1030.


\bibitem{CX2}
C. Xu, \emph{K-stability of Fano varieties: an algebro-geometric approach.} EMS Surv. Math. Sci. 8 (2021), no. 1-2, 265--354.

\bibitem{CX3}
C. Xu, \emph{K-stability of Fano varieties.} New Math. Monogr., 50. Cambridge University Press, Cambridge, 2025. xi+411 pp.

\bibitem{STY}
S. T. Yau, \emph{Open problems in geometry.} Differential geometry: partial differential equations on  manifolds ({L}os {A}ngeles, {CA}, 1990), Proc. Sympos. Pure Math., Vol. 54, 1--28.

\bibitem{ZZ}
J. Zhang and K. Zhou, \emph{On the geometry of non-collapsed polarized cscK surfaces}. arXiv:2606.02816, 49pp.

\end{thebibliography}
\end{document}